\documentclass[11pt, oneside]{amsart}
\usepackage[text={5.58in,8.5in},centering,letterpaper,dvips]{geometry}
\usepackage[dvipsnames]{xcolor}
\usepackage{graphicx}
\usepackage{subcaption}
\usepackage{amsfonts}
\usepackage{epsf}
\usepackage{amssymb}
\usepackage{amsmath}
\usepackage{amscd}
\usepackage{tikz}
\usepackage{pdfpages}
\usepackage{fancyhdr}
\usepackage{setspace}
\usepackage{hyperref}
\usepackage[all]{xy}
\usetikzlibrary{matrix}
\usepackage{verbatim}
\usepackage{enumerate}

\theoremstyle{theorem}
\newtheorem{theorem}{Theorem}[section]

\newtheorem{lemma}[theorem]{Lemma}
\newtheorem{question}[theorem]{Question}

\theoremstyle{definition}

\newtheorem{remark}[theorem]{Remark}

\newcommand{\U}{\mathcal{U}}
\newcommand{\Pp}{\mathcal{P}}

\newcommand{\Tt}{\mathcal T}

\newcommand{\TT}{\mathbb{T}}
\newcommand{\Ss}{\mathbb{S}}
\newcommand{\PP}{\mathbb{P}}

\newcommand{\pd}{\partial}

\newcommand{\be}{\begin{enumerate}}
\newcommand{\ee}{\end{enumerate}}

\newcommand{\K}{\mathcal{K}}

\newcommand{\T}{\mathcal T}

\newcommand{\nat}{\natural}
\newcommand{\sgn}{\text{sign}}

\makeatletter
\def\@seccntformat#1{%
  \protect\textup{\protect\@secnumfont
    \ifnum\pdfstrcmp{subsection}{#1}=0 \bfseries\fi
    \csname the#1\endcsname
    \protect\@secnumpunct
  }%
}  
\makeatother

\makeatletter
\newtheorem*{rep@theorem}{\rep@title}
\newcommand{\newreptheorem}[2]{%
\newenvironment{rep#1}[1]{%
 \def\rep@title{#2 \ref{##1}}%
 \begin{rep@theorem}}%
 {\end{rep@theorem}}}
\makeatother

\newreptheorem{theorem}{Theorem}
\newreptheorem{lemma}{Lemma}
\newreptheorem{question}{Question}
\newreptheorem{corollary}{Corollary}
\newreptheorem{proposition}{Proposition}

\begin{document}

\rhead{\thepage}
\lhead{\author}
\thispagestyle{empty}


\raggedbottom
\pagenumbering{arabic}
\setcounter{section}{0}


\title{Tri-plane diagrams for simple surfaces in $S^4$}

\author{Wolfgang Allred}
\address{University of Nebraska-Lincoln, Lincoln, NE 68588}
\email{wallred2@huskers.unl.edu}
\urladdr{}

\author{Manuel Arag\'on}
\address{Universidad de los Andes, Bogot\'a, Cundinamarca, Colombia}
\email{m.aragon@uniandes.edu.co}
\urladdr{}

\author{Zack Dooley}
\address{Reed College, Portland, OR 97202}
\email{dooleyz@reed.edu}
\urladdr{}

\author{Alexander Goldman}
\address{Skidmore College, Saratoga Springs, NY 12866}
\email{agoldma1@skidmore.edu}
\urladdr{}

\author{Yucong Lei}
\address{University of Michigan, Ann Arbor, MI 48109}
\email{leiyc@umich.edu}
\urladdr{}

\author{Isaiah Martinez}
\address{California State University, Fresno, Fresno, CA 93740}
\email{imartinez00@mail.fresnostate.edu}
\urladdr{}

\author{Nicholas Meyer}
\address{University of Nebraska-Lincoln, Lincoln, NE 68588}
\email{nicholas.meyer2@huskers.unl.edu}
\urladdr{https://nickmeyer.me/}

\author{Devon Peters}
\address{Oakland University, Rochester, MI 48309}
\email{devglemic@icloud.com}
\urladdr{}

\author{Scott Warrander}
\address{University of Edinburgh, Edinburgh EH8 Y9L, UK}
\email{swarrand@ed.ac.uk}
\urladdr{https://sites.google.com/view/scottwarrander/}

\author{Ana Wright}
\address{University of Nebraska-Lincoln, Lincoln, NE 68588}
\email{awright@huskers.unl.edu}
\urladdr{https://www.math.unl.edu/~awright14/}

\author{Alexander Zupan}
\address{University of Nebraska-Lincoln, Lincoln, NE 68588}
\email{zupan@unl.edu}
\urladdr{http://www.math.unl.edu/azupan2}

\begin{abstract}
Meier and Zupan proved that an orientable surface $\K$ in $S^4$ admits a tri-plane diagram with zero crossings if and only if $\K$ is unknotted, so that the crossing number of $\K$ is zero.  We determine the minimal crossing numbers of nonorientable unknotted surfaces in $S^4$, proving that $c(\Pp^{n,m}) = \max\{1,|n-m|\}$, where $\Pp^{n,m}$ denotes the connected sum of $n$ unknotted projective planes with normal Euler number $+2$ and $m$ unknotted projective planes with normal Euler number $-2$.  In addition, we convert Yoshikawa's table of knotted surface ch-diagrams to tri-plane diagrams, finding the minimal bridge number for each surface in the table and providing upper bounds for the crossing numbers.
\end{abstract}

\maketitle

\section{Introduction}

Tri-plane diagrams were introduced by Meier and Zupan in~\cite{MZB1} as an adaptation of the theory of trisections~\cite{GK} to the setting of knotted surfaces in $S^4$.  A \emph{tri-plane diagram} $D$ is a triple $(D_1,D_2,D_3)$ such that each $D_i$ is a planar diagram for a trivial tangle and $D_i \cup \overline{D}_j$ is an unlink diagram.  Meier and Zupan showed that any tri-plane diagram $D$ can be used to construct a surface $\K \subset S^4$, and conversely, every surface $\K \subset S^4$ can be represented in this way.

Tri-plane diagrams can be viewed as one (of many) natural ways to transfer ideas from classical knot theory to dimension four; for instance, crossing number has a natural analogue in this setting:  Define the \emph{crossing number} $c(D)$ of a tri-plane diagram to be
\[ c(D) = c(D_1) + c(D_2) + c(D_3),\]
and then define the \emph{crossing number} $c(\K)$ of $\K \subset S^4$ to be
\[ c(\K) = \min\{ c(D) : D \text{ is a tri-plane diagram representing } \K\}.\]
Proposition 4.4 of~\cite{MZB1} asserts that for an orientable surface $\K \subset S^4$, we have $c(\K) = 0$ if and only if $\K$ is unknotted.  In this paper, we determine the crossing numbers for unknotted nonorientable surfaces.  Let $\Pp^{\pm}$ denote the unknotted projective plane with normal Euler number $e(\Pp^{\pm}) = \pm 2$, and let $\Pp^{n,m}$ denote the connected sum of $n$ copies of $\Pp^+$ and $m$ copies of $\Pp^-$.  We prove
\begin{theorem}\label{main}
For any unknotted nonorientable surface $\Pp^{n,m}$,
\[ c(\Pp^{n,m}) = \max\{1,|n-m|\}.\]
\end{theorem}
The proof involves three key steps:  First, we show that if $\K \subset S^4$ is nonorientable, then $c(\K) > 0$.  Second, we prove that for any surface $\K \subset S^4$, we have $c(\K) \geq \frac{e(\K)}{2}$.  Finally, we realize diagrams for $\Pp^{n,m}$ with the claimed minimal crossing numbers.  Taken together, these three steps yield the main theorem.

Another type of diagram for knotted surfaces is called a \emph{ch-diagram}.  In~\cite{yoshi}, Yoshikawa published a table of ch-diagrams for the 23 simplest surfaces in 4-space, organized by a measure of complexity called their \emph{ch-index}.  We convert each of Yoshikawa's diagrams to a tri-plane diagram and record our findings in Table~\ref{table1}.  In addition to finding upper bounds for the crossing numbers of the surfaces in the table, we prove

\begin{theorem}\label{main2}
Each of the bridge numbers given in Table~\ref{table1} is minimal.
\end{theorem}

\subsection{Layout of the paper}
In Section~\ref{sec:prelim}, we provide the relevant background material, which we use to prove Theorem~\ref{main} in Section~\ref{sec:main1}.  Section~\ref{sec:convert} describes a method for converting ch-diagrams to tri-plane diagrams; in Section~\ref{sec:library}, we use this method to produce diagrams and data for the surfaces in Yoshikawa's table, proving Theorem~\ref{main2}.  Finally, Section~\ref{sec:questions} includes an assortment of questions to motivate future research.

\subsection{Acknowledgements}
This project was completed primarily during the summer of 2021 as part of the Polymath Jr. Virtual REU, and the authors are grateful to the Polymath Jr. organizers for providing the opportunity to carry out this research.  In addition, AZ thanks Jeffrey Meier and Maggie Miller for interesting conversation and appreciates the hospitality of the Max Planck Institute for Mathematics in Bonn, Germany, which accommodated him as a guest researcher during part of the completion of this work.  NM, AW, and AZ were supported by NSF grant DMS-2005518.

\section{Preliminaries}\label{sec:prelim}

\subsection{Tri-plane diagrams}

We work in the smooth category.  A \emph{surface} $\K$ in $S^4$ is a smoothly embedded closed 2-manifold, possibly nonorientable and possibly disconnected.  As noted above, a \emph{tri-plane diagram} $D$ is a triple $(D_1,D_2,D_3)$ such that each $D_i$ is a diagram for trivial $b$-stranded tangle, i.e. a diagram representing $b$ arcs such that each arc contains a single maximum point with respect to a natural height function on $D_i$, and such that for $i \neq j$, the classical link diagram $D_i \cup \overline{D}_j$ represented an unlink.  See Figure~\ref{fig:unknot} for examples.

\begin{figure}[h!]
\begin{subfigure}{.7\textwidth}
  \centering
  \includegraphics[width=0.5\linewidth]{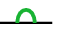}
  \label{fig:braid1}
  \caption{$U$\vspace{.4cm}}
\end{subfigure}
\begin{subfigure}{.35\textwidth}
  \centering
  \includegraphics[width=1\linewidth]{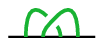}\vspace{-.4cm}
  \label{fig:braid1}
  \caption{$P^+$\vspace{.4cm}}
\end{subfigure}
\qquad \qquad \qquad
\begin{subfigure}{.35\textwidth}
  \centering
  \includegraphics[width=1\linewidth]{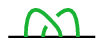}\vspace{-.4cm}
  \label{fig:braid1}
  \caption{$P^-$\vspace{.4cm}}
\end{subfigure}
\begin{subfigure}{.7\textwidth}
  \centering
  \includegraphics[width=1\linewidth]{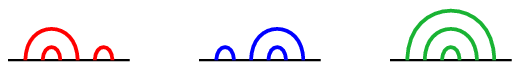}\vspace{-.4cm}
  \label{fig:braid1}
  \caption{$T$}
\end{subfigure}
	\caption{Tri-plane diagrams for the simplest unknotted surfaces}
\label{fig:unknot}
\end{figure}

Every tri-plane gives rise to a surface $\K \subset S^4$ in the following way:  Begin with the standard trisection of $S^4$, $X_1 \cup X_2 \cup X_3$, where each $X_i$ is a 4-ball and $B_i = X_i \cap X_{i-1}$ (indices taken modulo 3) is a 3-ball.  Embed the tangle represented by $D_i$ in $B_i$.  Then $\pd X_i = B_i \cup \overline{B}_{i+1}$ contains an unlink $U_i$ represented by the diagram $D_i \cup \overline{D}_{i+1}$, and we can attach disks $\mathcal D_i$ to $U_i$ in $X_i$.  The union $\mathcal D_1 \cup \mathcal D_2 \cup \mathcal D_3$ is then an embedded surface in $S^4$.  Moreover, the disks $\mathcal D_i$ are unique, and thus this process determines a unique surface, up to isotopy.  See~\cite{MZB1} for additional details.  Letting $c_i$ be the number of components in the unlink $U_i$, we will sometimes call this decomposition a $(b;c_1,c_2,c_3)$-bridge trisection.

In~\cite{MZB1}, it was proved that every surface $\K \subset S^4$ can be represented by a tri-plane diagram, and any two tri-plane diagrams $D$ and $D'$ are related by a finite sequence of moves:
\begin{enumerate}
\item Interior Reidemeister moves:  Classical Reidemeister moves performed on the interior of one of the tangle diagrams $D_i$.
\item Mutual braid transpositions:  A braid transposition $\sigma_j$ performed along the boundary of all three tangle diagrams $D_1$, $D_2$, and $D_3$.
\item Stabilization/destabilization moves:  A local move that increases or decreases the number of strands in each tangle diagram.
\end{enumerate}
An example of several mutual braid transpositions and interior Reidemeister moves being used to convert one diagram to another is shown in Figure~\ref{fig:11}.  By convention, we will depict mutual braid transpositions under the tri-plane diagram in gray, and strands that are omitted are assumed to carry the identity braid.

There are, in fact, an infinite family of stabilization/destabilization moves, but we will only need two of them for our purposes.  These two moves are depicted in Figures~\ref{fig:stab1} and~\ref{fig:stab2}, and they are also valid when performed on any permutation of the tangles or when reflected over a vertical line (as occurs in the 2-destabilization shown in Figure~\ref{fig:m}).  The curious reader is encouraged to reference~\cite{MZB1} for additional details.

\begin{figure}[h!]
  \centering
  \includegraphics[width=.40\linewidth]{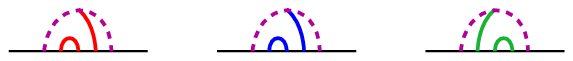}
  \raisebox{.2cm}{$\stackrel{\text{stab.}}{\longleftarrow}$   $\stackrel{\text{destab.}}{\longrightarrow}$}
  \includegraphics[width=.40\linewidth]{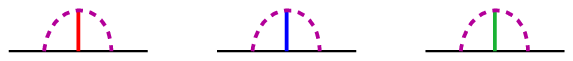}
	\caption{The 1-stabilization and 1-destabilization operations.}
\label{fig:stab1}
\end{figure}

\begin{figure}[h!]
  \centering
  \includegraphics[width=.40\linewidth]{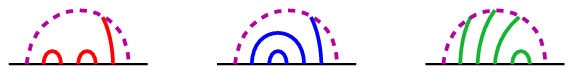}
 \raisebox{.2cm}{$\stackrel{\text{stab.}}{\longleftarrow}$   $\stackrel{\text{destab.}}{\longrightarrow}$}
  \includegraphics[width=.40\linewidth]{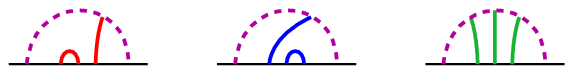}
	\caption{The 2-stabilization and 2-destabilization operations.}
\label{fig:stab2}
\end{figure}

We can also make sense of the connected sum operation on surfaces via tri-plane diagrams.  First, given a $b$-stranded trivial tangle diagram $D_i$ and a $b'$-stranded trivial tangle diagram $D_i'$, we define the \emph{boundary connected sum} of $D_i$ and $D_i'$, denoted $D_i \nat D_i'$, to be the tangle diagram constructed by attaching the rightmost endpoint of $D_i$ to the leftmost endpoint of $D_i'$.  Standard cut-and-paste arguments can be used to show that $D_i \nat D_i'$ is a $(b+b'-1)$-stranded trivial tangle diagram.  

Now, given tri-plane diagrams $D = (D_1,D_2,D_3)$ representing $\K \in S^4$ and $D' = (D_1',D_2',D_3')$ representing $\K'$ in $S^4$, define the \emph{connected sum} of $D$ and $D'$, denoted $D \# D'$, by
\[ D \# D' = (D_1 \nat D_1', D_2 \nat D_2', D_3 \nat D_3').\]
Then $D \# D'$ is a tri-plane diagram for the surface $\K \# \K'$ in $S^4$.  We note that as is the case with classical knots, the connected sum operation is commutative, associative, and does not depend on the choice of points in $\K$ and $\K'$ at which the operation takes place.  The top frame of Figure~\ref{fig:11} shows an example of the diagram $P^+ \# P^-$.  

\subsection{Invariants via tri-plane diagrams}

We can measure the complexity of a tri-plane diagram $D$ in several different ways.  The \emph{crossing number} $c(D)$ is the sum $c(D_1) + c(D_2) + c(D_3)$, while the \emph{bridge number} $b(D)$ is the number of strands in each of the tangles $D_i$.  To obtain surface invariants, we minimize over all possible diagrams:
\begin{eqnarray*}
c(\K) &=& \min\{ c(D) : D \text{ is a tri-plane diagram representing } \K\},\\
b(\K) &=& \min\{ b(D) : D \text{ is a tri-plane diagram representing } \K\}.
\end{eqnarray*}

There are other types of invariants better suited to distinguishing surfaces, and we can compute these invariants from a single diagram.  For example, the \emph{Euler characteristic} of a surface $\K$ admitting a tri-plane diagram $D$ inducing a $(b;c_1,c_2,c_3)$-bridge trisection is given by
\[ \chi(\K) = c_1+c_2+c_3 - b,\]
as shown in~\cite{MZB1}.  Another such invariant is orientability.  An \emph{orientation} of a tri-plane diagram $D$ is an assignment of a $+$ or a $-$ to each of the $2b$ endpoints of the three tangles (called \emph{bridge points}) so that the endpoints of $D_1$, $D_2$, and $D_3$ have consistent labels, and each strand in $D_i$ connects a $+$ and $-$ bridge point.  The tri-plane diagram is \emph{orientable} if it admits an orientation.  It was shown in~\cite{MTZ} that $D$ is orientable if and only if the surface $\K$ is orientable.

\begin{lemma}\label{nonor}
If $c(\K) = 0$, then $\K$ is orientable.
\end{lemma}

\begin{proof}
Suppose $D$ is a zero-crossing tri-plane diagram for $\K$, with bridge points $x_1,\dots,x_{2b}$ labeled in order from left to right.  For each odd $i$, assign $x_i$ a $+$; for even $i$, assign $x_i$ a $-$.  Suppose that an arc $a$ in $D_j$ has two endpoints labeled $+$ or two endpoints labeled $-$.  Then $a$ encloses an odd number of bridge points, implying that some other arc $a'$ in $D_j$ must cross $a$, a contradiction.  It follows that $D$, and thus $\K$, is orientable.
\end{proof}

For the other invariant we will study here, we need a preliminary definition:  Given an oriented link diagram $D$, the \emph{writhe} of $D$ is the signed count of the crossings of $D$, with standard conventions for positive and negative crossings shown in Figure~\ref{fig:writhe}.  Now, let $D$ be a tri-plane diagram.  The \emph{normal Euler number} of $D$, denoted by $e(D)$, is given by
\[ e(D) = w(D_1 \cup \overline{D}_2) + w(D_2 \cup \overline{D}_3) + w(D_3 \cup \overline{D}_1).\]
In~\cite{JMMZ}, the authors proved that for any two diagrams $D$ and $D'$ for $\K$, we have $e(D) = e(D')$, and so $e(\K) = e(D)$ is a knotted surface invariant; in fact, this definition agrees with the classical definition of \emph{normal Euler number} (see~\cite{JMMZ} for further details regarding the normal Euler number).

\begin{figure}[h!]
\begin{subfigure}{.4\textwidth}
  \centering
  \includegraphics[width=.2\linewidth]{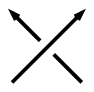}
  \label{fig:braid1}
  \caption{$+$}
\end{subfigure}
\begin{subfigure}{.4\textwidth}
  \centering
  \includegraphics[width=.2\linewidth]{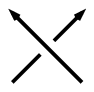}
  \label{fig:braid2}
  \caption{$-$}
\end{subfigure}
	\caption{Conventions for computing writhe}
\label{fig:writhe}
\end{figure}

 Note that for a knot diagram $D$, the writhe $w(D)$ does not depend on the choice of orientation.  In general, for a link diagram $D$, the writhe $w(D)$ \emph{does} depend on orientation choices, but only if components have nonzero linking numbers.  For tri-plane diagrams, $D_i \cup \overline{D_j}$ is an unlink diagram (pairwise linking numbers are zero), and so the computation $w(D_i \cup \overline{D}_{i+1})$ is independent of any of the chosen orientations.

Crossing number and normal Euler number can be related by the following inequality:

\begin{lemma}\label{euler}
For a knotted surface $\K$, we have $|e(\K)| \leq 2 c(\K)$.
\end{lemma}

\begin{proof}
Let $D =(D_1,D_2,D_3)$ be a tri-plane diagram for $\K$ such that $c(D) = c(\K)$.  In addition, suppose $D_i$ has $n_i$ crossings, so $c(D) = n_1 + n_2 + n_3$.  Let $Q_i$ be the set of crossings of $D_i \cup \overline{D}_{i+1}$.  Then $|Q_i| = n_i + n_{i+1}$, and choosing an orientation for $D_i \cup \overline{D}_{i+1}$, we have $w(D_i \cup \overline{D}_{i+1}) = \sum_{q \in Q_i} \sgn(q)$.  We compute
\begin{eqnarray*}
|e(\K)| = |e(D)| &=& \left| \sum_{q \in Q_i} \sgn(q) + \sum_{q \in Q_2} \sgn(q) + \sum_{q \in Q_3} \sgn(q) \right| \\
&\leq& \sum_{q \in Q_1} |\sgn(q)| + \sum_{q \in Q_2} |\sgn(q)| + \sum_{q \in Q_3} |\sgn(q)| \\
&=& |Q_1| + |Q_2| + |Q_3| \\
&=& 2 c(D) \\
&=& 2 c(\K).
\end{eqnarray*}
\end{proof}

We can also examine how crossing number, bridge number, and normal Euler number behave under connected sum.  By construction, the diagram $D \# D'$ satisfies
\[ c(D \# D') = c(D) + c(D'), \quad b(D \# D') = b(D) + b(D') - 1, \quad \text{and} \quad e(D \# D') = e(D) + e(D').\]
Since $e(\K) = e(D)$ for any diagram $D$ for $\K$, it follows that
\[ e(\K \# \K') = e(\K) + e(\K').\]
However, this does not necessarily imply that similar equalities hold for crossing number or bridge number; choosing minimal diagrams $D$ and $D'$ yields the inequalities $c(\K \# \K') \leq c(\K) + c(\K')$ and $b(\K \# \K') \leq b(\K) + b(\K') - 1$, but (as we will discuss below in Remark~\ref{degen}), there are cases in which $\K \# \K'$ has other diagrams, not realized by the above construction, that achieve lower values for both crossing number and bridge number; hence, the inequalities can be strict.

It is a fact that every tri-plane diagram $D$ is equivalent to a tri-plane diagram $D' = (D_1',D_2',D_3')$ such that every crossing in $D'$ is contained in a single tangle, say $D_3'$.  We call such a diagram \emph{concentrated}.  For instance, the diagrams for $P^{\pm}$ in Figure~\ref{fig:unknot} are concentrated diagrams.

\subsection{Unknotted surfaces}

Much of this paper concerns surfaces in $S^4$ that are unknotted, which we define here.  To begin, consider the surfaces $\U$, $\Pp^+$, $\Pp^-$, and $\Tt$ defined by the tri-plane diagrams $U$, $P^+$, $P^-$, and $T$ shown in Figure~\ref{fig:unknot}.  We call $\U$ the \emph{unknotted 2-sphere}, $\Pp^{\pm}$ the \emph{positive/negative unknotted projective plane}, and $\Tt$ the \emph{unknotted torus}.

A surface $\K \subset S^4$ is called \emph{unknotted} if $\K$ is $\U$, $\K$ is $\Tt^g$ (the connected sum of $g$ copies of $\T$), or $\K$ is $\Pp^{n,m}$ (the connected sum of $n$ copies of $\Pp^+$ and $m$ copies of $\Pp^-$).  This definition agrees with classical notions of unknottedness; see~\cite{HK79} and~\cite{MTZ} for further details about unknotted surfaces.  The surfaces $\U$ and $\T^g$ are orientable, while $\Pp^{n,m}$ is nonorientable.  In addition, using the fact that $e(\Pp^{\pm}) = \pm 2$ and the formula for the behavior of normal Euler number under connected sum, we have
\[ e(\Pp^{n,m}) = 2(n-m).\]

\section{The crossing numbers of unknotted surfaces}\label{sec:main1}

As noted in the introduction, we prove the main theorem by establishing lower and upper bounds for $c(\Pp^{n,m})$.  Lemmas~\ref{nonor} and~\ref{euler} establish the necessary lower bounds.  For the upper bounds, we prove several additional lemmas.

\begin{lemma}\label{equal}
For each $n > 0$, there exists a 1-crossing diagram $P^{n,n}$ for $\Pp^{n,n}$.
\end{lemma}

\begin{proof}
First, consider the diagram $P^+ \# P^-$.  By performing mutual braid moves and interior Reidemeister moves as shown in Figure~\ref{fig:11}, we obtain the 1-crossing diagram $P^{1,1}$.  Now, let $P^{n,n}$ be the 1-crossing diagram shown at bottom Figure~\ref{fig:nn}.  Inducting on $n$, we will show that $P^{n,n}$ is a diagram for the surface $\Pp^{n,n}$.  The base case has already been completed.  Suppose by way of induction that $P^{n-1,n-1}$ is a diagram for $\Pp^{n-1,n-1}$.  Then $P^{n-1,n-1} \# P^+ \# P^-$ is a diagram for $\Pp^{n,n}$, and by performing mutual braid moves and interior Reidemeister moves as shown at top in Figure~\ref{fig:nn}, we can convert $P^{n-1,n-1} \# P^+ \# P^-$ into $P^{n,n}$, completing the proof.  (In fact, we note that the diagrams in Figure~\ref{fig:11} are precisely the diagrams in Figure~\ref{fig:nn} in the case $n = 1$, where $P^{0,0}$ here is the diagram $U$.)
\end{proof}

\begin{figure}[h!]
\begin{subfigure}{.7\textwidth}
  \centering
  \includegraphics[width=1\linewidth]{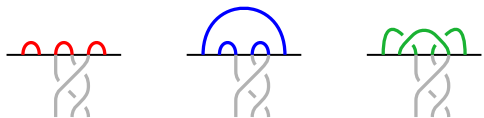}
  \label{fig:11a}
  \caption{The diagram $P^+ \# P^-$\vspace{.4cm}}
\end{subfigure}
\begin{subfigure}{.7\textwidth}
  \centering
  \includegraphics[width=1\linewidth]{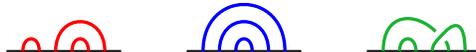}\vspace{-.4cm}
  \label{fig:11b}
  \caption{The diagram $P^{1,1}$}
\end{subfigure}
	\caption{Obtaining $P^{1,1}$ from $P^+ \# P^-$}
\label{fig:11}
\end{figure}

\begin{figure}[h!]
\begin{subfigure}{.8\textwidth}
  \centering
  \includegraphics[width=1\linewidth]{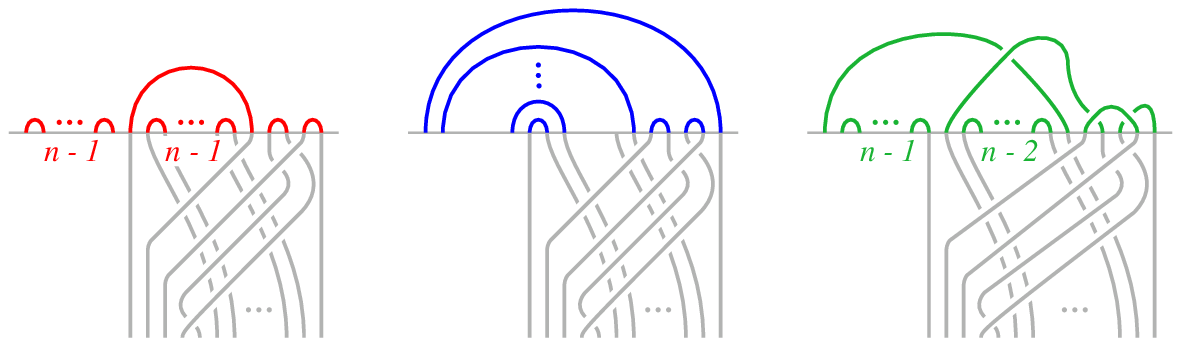}
  \label{fig:11a}
  \caption{The diagram $P^{n-1,n-1} \# P^+ \# P^-$\vspace{.4cm}}
\end{subfigure}
\begin{subfigure}{.8\textwidth}
  \centering
  \includegraphics[width=1\linewidth]{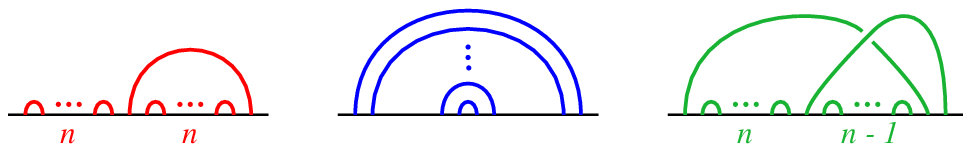}\vspace{-.4cm}
  \label{fig:11b}
  \caption{The diagram $P^{n,n}$}
\end{subfigure}
	\caption{Obtaining $P^{n,n}$ from $P^{n-1,n-1} \# P^+ \# P^-$}
\label{fig:nn}
\end{figure}

We also need another family of diagrams.

\begin{lemma}\label{plusone}
For each $n \geq 0$, there exists a 1-crossing diagram $P^{n+1,n}$ for $\Pp^{n+1,n}$.
\end{lemma}

\begin{proof}
When $n=0$, we let $P^{1,0} = P^+$, which has one crossing.  Now, suppose that $n > 0$.  By Lemma~\ref{equal}, $P^{n,n}$ is a diagram for $\Pp^{n,n}$, and thus $P^{n,n} \# P^+$ is a diagram for $\Pp^{n+1,n}$.  Using mutual braid moves and interior Reidemeister moves, we can convert $P^{n,n} \# P^+$ into the 1-crossing diagram $P^{n+1,n}$ shown in Figure~\ref{fig:plusone}.
\end{proof}

\begin{figure}[h!]
\begin{subfigure}{.8\textwidth}
  \centering
  \includegraphics[width=1\linewidth]{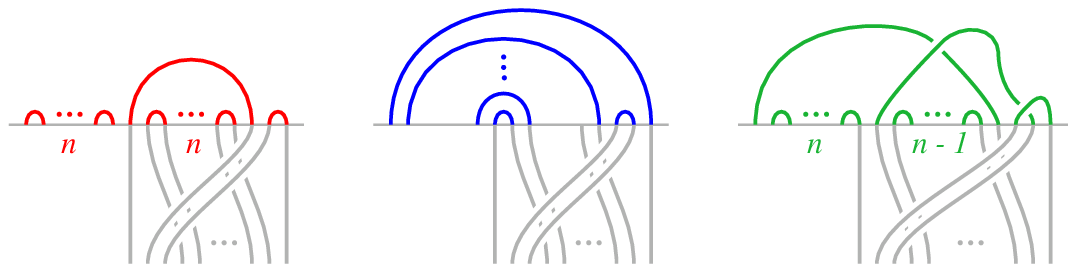}
  \label{fig:11a}
  \caption{The diagram $P^{n,n} \# P^+$\vspace{.4cm}}
\end{subfigure}
\begin{subfigure}{.8\textwidth}
  \centering
  \includegraphics[width=1\linewidth]{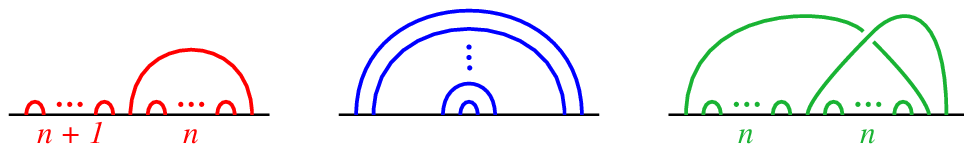}\vspace{-.4cm}
  \label{fig:11b}
  \caption{The diagram $P^{n+1,n}$}
\end{subfigure}
	\caption{Obtaining $P^{n+1,n}$ from $P^{n,n} \# P^+$}
\label{fig:plusone}
\end{figure}

We now proceed to the proof of the first main theorem.

\begin{proof}[Proof of Theorem~\ref{main}]
Consider $\Pp^{n,m}$, and note that $\Pp^{m,n}$ is the mirror image of $\Pp^{n,m}$, so that $c(\Pp^{n,m}) = c(\Pp^{m,n})$.  Thus, we may suppose without loss of generality that $n \geq m$.  Since $\Pp^{n,m}$ is nonorientable, Lemma~\ref{nonor} asserts that $c(\Pp^{n,m}) \geq 1$.  Additionally, Lemma~\ref{euler} implies that
\[c(\Pp^{n,m}) \geq \frac{1}{2}e(\Pp^{n,m}) = |n-m|.\]
Suppose first that $n=m$.  By Lemma~\ref{equal}, $c(\Pp^{n,n})$ is at most one, and thus
\[ c(\Pp^{n,n}) = 1.\]
On the other hand, suppose $n > m$, and let $j = n-m - 1$, so that $j \geq 0$.  Then we can express $\Pp^{n,m}$ as the connected sum of $\Pp^{n+1,n}$ and $j$ copies of $\Pp^+$, and a diagram for $\Pp^{n,m}$ can be obtained by taking the connected sum of the 1-crossing diagram $P^{n+1,n}$ from Lemma~\ref{plusone} and $j$ copies of the diagram $P^+$, which has $1+ j = n-m$ crossings in total.  It follows that $c(\Pp^{n,m}) \leq n-m$.  Taken together with the inequality above, we have
\[ c(\Pp^{n,m}) = n-m.\]
In every case, we conclude that the desired equality holds; that is, for any values of $n$ and $m$, we have
\[ c(\Pp^{n,m}) = \max\{1,|n-m|\}.\]
\end{proof}

\begin{remark}\label{degen}
It follows from the main theorem that, for example,
\[ c(\Pp^+ \# \Pp^-) < c(\Pp^+) + c(\Pp^-),\]
and so crossing number is not additive under the connected sum operation.  Even more strongly, for every $n$, there exist surfaces $\K = \Pp^{n,0}$ and $\K' = \Pp^{0,n}$ such that $c(\K) = c(\K') = n$, but $c(\K \# \K') = 1$.

Turning our attention to bridge number, there is an example due to Viro of a knotted 2-sphere $\K$ such that $\Pp^+ \# \K = \Pp^+$~\cite{viro}.  By~\cite{MZB1} and~\cite{MZ}, the we have $b(\K) \geq 4$, and thus it is known that bridge number is also not additive under connected summation.  There are no known examples of such degeneration for either invariant when the operation is restricted to the class of orientable surfaces, but finding such examples would be an interesting avenue of future research.
\end{remark}

\begin{remark}
An alternative proof of Lemmas~\ref{equal} and~\ref{plusone} uses the fact that $\Pp^+ \# \T = \Pp^+ \# \Pp^+ \# \Pp^-$, and so for $n >1$, we have $\Pp^{n,n} = \T^{n-1} \# \Pp^{1,1}$, and $\Pp^{n+1,n} = \Pp^+ \# T^n$, and the corresponding diagrams have a single crossing.  The above proofs are interesting because they produce diagrams $P^{n,n}$ and $P^{n+1,n}$ that are potentially inequivalent to diagrams that decompose with $T$ summands.  Indeed, Jeffrey Meier has shown that the diagrams $P^+ \# T$ and $P^+ \# P^+ \# P^-$ correspond to inequivalent bridge trisections by examining their underlying cubic graphs~\cite{jeffPersonal}.
\end{remark}

\section{Converting ch-diagrams to tri-plane diagrams}\label{sec:convert}

Another tool to encode knotted surfaces in 4-space is a \emph{ch-diagram}, which appeared in work of Yoshikawa~\cite{yoshi}.  Whereas a classical knot diagram is an immersed curve in the plane with crossing information at each double point, ch-diagrams contain both crossings and \emph{marked vertices}, as shown in Figure~\ref{fig:marked}.  Marked vertices admit a $+$ and $-$ resolution, as shown.

\begin{figure}[h!]
  \centering
  \includegraphics[width=.07\linewidth]{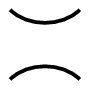} \quad
  \raisebox{0.4cm}{$\stackrel{\text{+}}{\longleftarrow}$} \quad
  \includegraphics[width=.07\linewidth]{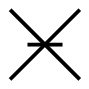} \quad
  \raisebox{0.4cm}{$\stackrel{\text{-}}{\longrightarrow}$} \quad
  \includegraphics[width=.07\linewidth]{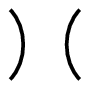}
	\caption{A marked vertex (center) along with its two resolutions}
\label{fig:marked}
\end{figure}

A \emph{marked diagram} $M$ is an immersed curve or curves in the plane such that each double point is either a crossing or a marked vertex.  As such, a marked diagram has a $+$ resolution $M_+$, obtained by performing the $+$ resolution on all marked vertices, and a $-$ resolution $M_-$, obtained by performing the $-$ resolution on all marked vertices.  Note that $M_+$ and $M_-$ are classical link diagrams.  Finally, a \emph{ch-diagram} is defined to be a marked diagram $M$ such that each resolution $M_{\pm}$ is a classical diagram for an unlink, a link isotopic to the disjoint union of unknotted loops.  As is the case with tri-plane diagrams, every surface in 4-space admits a ch-diagram, and any two ch-diagrams are related by a finite collection of moves, called \emph{Yoshikawa moves}~\cite{yoshi,swenton,kearton-kurlin}.  An example of a ch-diagram for the surface $\Pp^+$ (denoted $2^{-1}_1$ in Yoshikawa's table) is shown in Figure~\ref{fig:ex}.

To convert a ch-diagram to a tri-plane diagram, we (implicitly) pass through yet another closely related concept, called a \emph{banded unlink diagram}.  Banded unlink diagrams are described in~\cite{MZB1}; we refer the reader to that paper for further details.  Briefly, every marked vertex can be converted to a band as in Figure~\ref{fig:band}, and this process changes the ch-diagram $M$ to a banded link diagram $(M_-,v)$, where the resolution of the unlink $M_-$ along the bands $v$ yields the unlink $M_+$.  

\begin{figure}[h!]
  \centering
  \includegraphics[width=.07\linewidth]{figures/marked.eps} \quad
  \raisebox{0.4cm}{$\longrightarrow$} \quad
  \includegraphics[width=.07\linewidth]{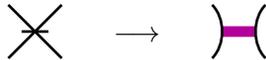}
	\caption{Changing a marked vertex to a band}
\label{fig:band}
\end{figure}

We say that a ch-diagram $M$ is in \emph{bridge position} if there is some height function on $M$ such that all maxima of $M$ occur above all minima, all markings on marked vertices are contained in a single (planar) regular level $P$ separating the minima and the maxima, and the arcs containing the maxima can be isotoped into $P$ to be disjoint from the marked vertices and such that the image of the isotopy is a collection of embedded arcs (with no crossings and no closed components).  In this case, the corresponding banded link presentation $(M_-,v)$ admits a \emph{banded bridge splitting}, which gives rise to a bridge trisection by Lemma 3.2 of~\cite{MZB1}.

A ch-diagram $M$ in bridge position yields a tri-plane diagram via the following procedure:  Let $P'$ be a regular level just below the marked vertices, cutting $M$ into two diagrams $M'$ and $M''$, where $M'$ contains the marked vertices.  Letting $D_1 = M'_-$, $D_2 = M'_+$, and $D_3 = \overline{M''}$, it follows from~\cite{MZB1} that $D = (D_1,D_2,D_3)$ is a tri-plane diagram for the knotted surface $\K$ determined by $M$.  An example of this process using Yoshikawa's ch-diagram $2^{-1}_1$ for $\Pp^+$ is carried out in Figure~\ref{fig:ex}, including tri-plane moves showing that the converted diagram is equivalent to the standard diagram $P^+$.

\begin{figure}[h!]
\begin{subfigure}{.6\textwidth}
  \centering
  \includegraphics[width=.175\linewidth]{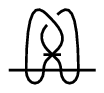}\vspace{-.2cm}
  \label{fig:braid1}
  \caption{ch-diagram in bridge position\vspace{.4cm}}
\end{subfigure}
\begin{subfigure}{.7\textwidth}
  \centering
  \includegraphics[width=0.5\linewidth]{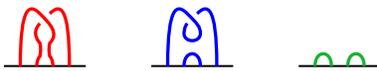}\vspace{-0.2cm}
  \label{fig:braid2}
  \caption{Result of converting ch-diagram to tri-plane diagram\vspace{0.4cm}}
\end{subfigure}
\begin{subfigure}{.7\textwidth}
  \centering
  \includegraphics[width=0.5\linewidth]{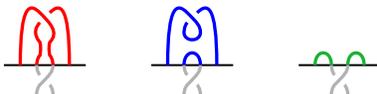}\vspace{-.2cm}
  \label{fig:braid3}
  \caption{Braid and Reidemeister moves convert the diagram to $P^+$}
\end{subfigure}
	\caption{Converting a ch-diagram to a tri-plane diagram}
\label{fig:ex}
\end{figure}

\section{A library of tri-plane diagrams}\label{sec:library}

In this section, we carry out the procedure described above to convert the 23 ch-diagrams given by Yoshikawa in~\cite{yoshi} into tri-plane diagrams.  For some diagrams, we perform tri-plane moves to decrease the crossing numbers.  We begin with a class of knotted 2-spheres introduced by Artin called \emph{spun knots}~\cite{Artin}, since two of the surfaces ($8_1$ and $10_1$ in the Yoshikawa table) fall into this family.  Tri-plane diagrams for spun knots appeared in~\cite{MZB1}, in which the authors determined the minimal bridge number, but we reproduce the argument below to prove the next lemma.

\begin{lemma}\label{spun}
Suppose that $\K$ is the spin of a 2-bridge knot $K$.  Then $b(\K) = 4$ and $c(\K) \leq 2 c(K)$.
\end{lemma}

\begin{proof}
Every 2-bridge knot $K$ has a minimal crossing diagram obtained by taking the plat closure of a braid on four strands, and such that the right-most strand contains no crossings.  As such, $\K$ has a ch-diagram as shown in Figure~\ref{fig:braid}.  Converting this ch-diagram to a tri-plane diagram and performing braids moves and Reidemeister moves produces the tri-plane diagram at bottom in Figure~\ref{fig:braid}, which has $2 c(K)$ crossings.  Thus, $c(\K) \leq 2 c(K)$.  Since $\K$ is not the unknotted 2-sphere, we have that $b(\K) \geq 4$ by~\cite{MZ}, and thus $b(\K) = 4$.
\end{proof}

\begin{figure}[h!]
\begin{subfigure}{1\textwidth}
  \centering
  \includegraphics[width=.21\linewidth]{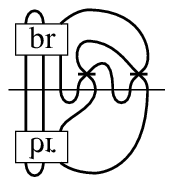}
  \label{fig:braid1}
  \caption{ch-diagram, where \emph{br} represents a braid (and reflected \emph{br} the reflection of the braid) \vspace{.4cm}}
\end{subfigure}
\begin{subfigure}{.7\textwidth}
  \centering
  \includegraphics[width=1\linewidth]{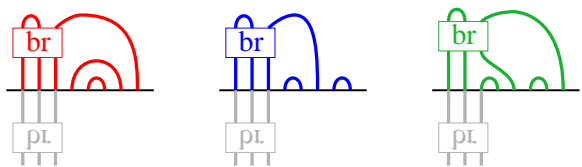}\vspace{-0.4cm}
  \label{fig:braid2}
  \caption{Result of converting ch-diagram to tri-plane diagram\vspace{0.4cm}}
\end{subfigure}
\begin{subfigure}{.7\textwidth}
  \centering
  \includegraphics[width=1\linewidth]{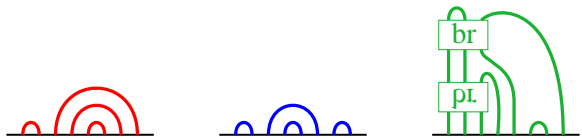}\vspace{-.4cm}
  \label{fig:braid3}
  \caption{Crossing number lowered by braid and Reidemeister moves}
\end{subfigure}
	\caption{A tri-plane diagram for a spun knot $\K$}
\label{fig:braid}
\end{figure}

In the table below, we collect data for each knotted surface in Yoshikawa's table, with the associated tri-plane diagrams appearing in the referenced figures.  While the each entry for the crossing number is only an upper bound, the Theorem~\ref{main2} verifies that each bridge number is minimal.  In addition, the \emph{type} of a surface is its homeomorphism class as a 2-manifold; we use $\Ss^2$, $\PP^2$, and $\TT^2$ to denote the 2-sphere, projective plane, and torus, respectively.

\begin{proof}[Proof of Theorem~\ref{main2}]
We obtain lower bounds for most of the bridge numbers by considering the bridge numbers of individual components of $\K$ separately.   For example, if the type of $\K$ is $\Ss^2 \sqcup \TT^2$, we know that the $\Ss^2$ component on its own has bridge number of at least one, while the $\TT^2$ component on its own has bridge number of at least three, and thus the bridge number of $\K$ is at least four.  Similarly, each $\PP^2$ component contributes bridge number two to the sum, and each $\PP^2 \# \PP^2$ component contributes bridge number three to the sum.  Each surface $\K$ in the table realizes the sum of minimum possible bridge numbers of its components, with the exception of the surfaces $8_1$, $9_1$, $10_1$, $10_2$, $10_3$, and $10^1_1$.

Each of the first five surfaces in this list is a 2-sphere $\K$ that is not the unknotted 2-sphere, and since $b(\K) \geq 4$ by~\cite{MZ}, we have $b(\K) = 4$ for these surfaces.  For the remaining surface $10^1_1$, we claim that $b(10^1_1) = 6$.  In Figure~\ref{fig:j}, we see a 6-bridge tri-plane diagram for $10^1_1$, and so certainly $b(10^1_1) \leq 6$.  If $10^1_1$ admits some $(b;c_1,c_2,c_3)$-bridge trisection with $b < 6$, then using the fact that $0 = \chi(10^1_1) = c_1+c_2+c_3 - b$, we have that at least one of the $c_i$'s is equal to one.  By Theorem 4.1 of~\cite{JMMZ}, it follows that $\pi_1(S^4 \setminus 10^1_1)$ admits a presentation with a single generator, and as such it must be a cyclic group.  On the other hand, in~\cite{yoshi}, it is noted that $\pi_1(S^4 \setminus 10^1_1)$ is the same as the fundamental group of the exterior of the trefoil in 3-space, the group $\langle x,y \, | \, xyx = yxy \rangle$, which is neither cyclic nor abelian, yielding a contradiction.  We conclude that $b(10^1_1) = 6$, completing the proof.
\end{proof}

In Table~\ref{table1} below, we record the name (from~\cite{yoshi}), the figure depicting a tri-plane diagram, the topological type, the bridge number, an upper bound for the crossing number, and the Euler number for each surface.  We note that for all surfaces in the table with more than one component, the individual components are unknotted.  The crossing number entry for $10^{0,-2}_2$ is marked with an $*$ because the diagram which minimizes crossing number experimentally has bridge number equal to five, instead of the minimal bridge number for this surface, which is four.  In addition, recall that a tri-plane diagram is \emph{concentrated} if all crossings are contained in a single tangle.  When a diagram which minimizes crossing number is concentrated, we mark the corresponding crossing number with a $c$.

\begin{table}[ht]
\centering
\caption{Tri-plane data for the Yoshikawa table}
\label{table1}
\begin{tabular}{lllllll}
\hline
Label & Figure & Type & Bridge \# & Crossing \# $\leq$ & Euler \#  \\
\hline \hline
$0_1$ & \ref{fig:unknot} & $\Ss^2$ & 1 & $0^c$ & 0 \\
\hline
$2^1_1$ & \ref{fig:unknot} & $\TT^2$ & 1 & $0^c$ & 0 \\
\hline
$2^{-1}_1$ & \ref{fig:unknot} & $\PP^2$ & 2 & $1^c$ & 2 \\
\hline
$6^{0,1}_1$ & \ref{fig:a} & $\Ss^2 \sqcup \TT^2$ & 4 & $4^c$ & 0 \\
\hline
$7^{0,-2}_1$ & \ref{fig:b} & $\Ss^2 \sqcup (\PP^2 \# \PP^2)$ & 4 & $5^c$ & 0 \\
\hline
$8_1$ & \ref{fig:braid} & $\Ss^2$ & 4 & $6^c$ & 0 \\
\hline
$8^{1,1}_1$ & \ref{fig:c} & $\TT^2 \sqcup \TT^2$ & 6 & 6 & 0 \\
\hline
$8^{-1,-1}_1$ & \ref{fig:d} & $\PP^2 \sqcup \PP^2$ & 4 & 8 & 0 \\
\hline
$9_1$ & \ref{fig:e} & $\Ss^2$ & 4 & $7^c$ & 0 \\
\hline
$9^{0,1}_1$ & \ref{fig:f} & $\Ss^2 \sqcup \TT^2$ & 4 & $8^c$ & 0 \\
\hline
$9^{1,-2}_1$ & \ref{fig:g} & $\TT^2 \sqcup (\PP^2 \# \PP^2)$ & 6 & 7 & 0 \\
\hline
$10_1$ & \ref{fig:braid} & $\Ss^2$ & 4 & $8^c$ & 0 \\
\hline
$10_2$ & \ref{fig:h} & $\Ss^2$ & 4 & 11 & 0 \\
\hline
$10_3$ & \ref{fig:i} & $\Ss^2$ & 4 & 11 & 0 \\
\hline
$10^1_1$ & \ref{fig:j} & $\TT^2 $ & 6 & 11 & 0 \\
\hline
$10^{0,1}_1$ & \ref{fig:k} & $\Ss^2 \sqcup \TT^2$ & 4 & 8 & 0 \\
\hline
$10^{0,1}_2$ & \ref{fig:l} & $\Ss^2 \sqcup \TT^2$ & 4 & $8^c$ & 0 \\
\hline
$10^{1,1}_1$ & \ref{fig:m} & $\TT^2 \sqcup \TT^2$ & 6 & 6 & 0 \\
\hline
$10^{0,0,1}_1$ & \ref{fig:n} & $\Ss^2 \sqcup \Ss^2 \sqcup \TT^2 $ & 5 & 8 & 0 \\
\hline
$10^{0,-2}_1$ & \ref{fig:o} & $\Ss^2 \sqcup (\PP^2 \# \PP^2) $ & 4 & $8^c$ & 0 \\
\hline
$10^{0,-2}_2$ & \ref{fig:p} & $\Ss^2 \sqcup (\PP^2 \# \PP^2) $ & 4 & $8^{\ast,c}$ & 0 \\
\hline
$10^{-1,-1}_1$ & \ref{fig:q} & $\PP^2 \sqcup \PP^2$ & 4 & $8^c$ & 0 \\
\hline
$10^{-2,-2}_1$ & \ref{fig:r} & $(\PP^2 \# \PP^2) \sqcup (\PP^2 \# \PP^2)$ & 6 & $8$ & 0 \\
\hline
\end{tabular}
\end{table}

\begin{figure}[h!]
\begin{subfigure}{.3\textwidth}
  \centering
  \includegraphics[width=.7\linewidth]{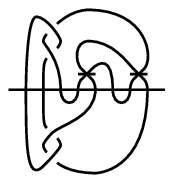}
  \label{fig:a1}
  \caption{ch-diagram\vspace{.4cm}}
\end{subfigure}
\begin{subfigure}{.7\textwidth}
  \centering
  \includegraphics[width=1\linewidth]{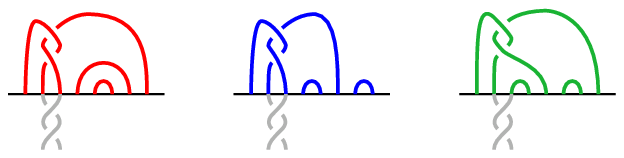}\vspace{-0.4cm}
  \label{fig:a2}
  \caption{Result of converting ch-diagram to tri-plane diagram\vspace{0.4cm}}
\end{subfigure}
\begin{subfigure}{.7\textwidth}
  \centering
  \includegraphics[width=1\linewidth]{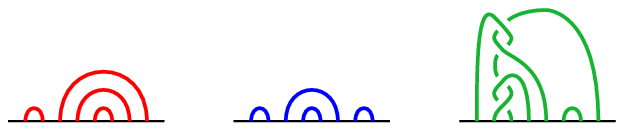}\vspace{-.4cm}
  \label{fig:a3}
  \caption{Crossing number lowered by braid and Reidemeister moves}
\end{subfigure}
	\caption{$6^{0,1}_1$}
\label{fig:a}
\end{figure}

\begin{figure}[h!]
\begin{subfigure}{.3\textwidth}
  \centering
  \includegraphics[width=.7\linewidth]{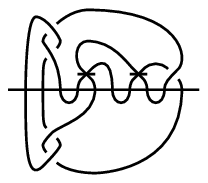}
  \label{fig:b1}
  \caption{ch-diagram\vspace{.4cm}}
\end{subfigure}
\begin{subfigure}{.7\textwidth}
  \centering
  \includegraphics[width=1\linewidth]{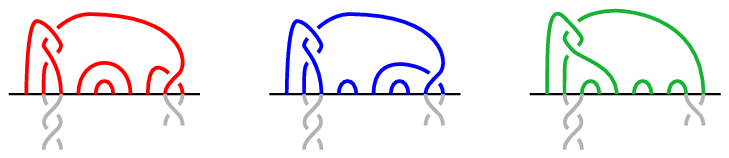}\vspace{-.4cm}
  \label{fig:b2}
  \caption{Result of converting ch-diagram to tri-plane diagram \vspace{.4cm}}
\end{subfigure}
\begin{subfigure}{.7\textwidth}
  \centering
  \includegraphics[width=1\linewidth]{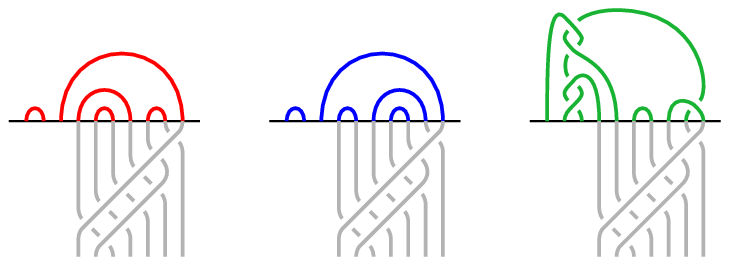}\vspace{-.4cm}
  \label{fig:b4}
  \caption{Crossing number lowered by braid and Reidemeister moves\vspace{.4cm}}
\end{subfigure}
\begin{subfigure}{.7\textwidth}
  \centering
  \includegraphics[width=1\linewidth]{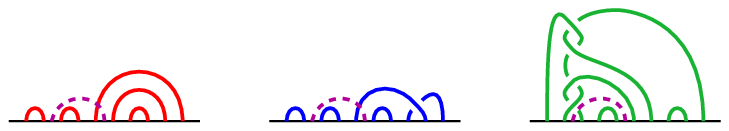}\vspace{-.4cm}
  \label{fig:b5}
  \caption{Braid and Reidemeister moves set up a destabilization\vspace{.4cm}}
\end{subfigure}
\begin{subfigure}{.7\textwidth}
  \centering
  \includegraphics[width=1\linewidth]{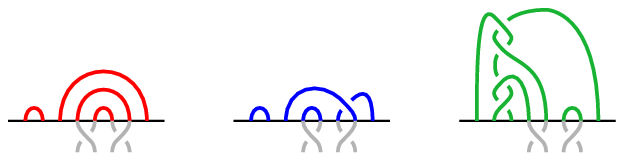}\vspace{-.4cm}
  \label{fig:b6}
  \caption{Result of destabilization\vspace{.4cm}}
\end{subfigure}
\begin{subfigure}{.7\textwidth}
  \centering
  \includegraphics[width=1\linewidth]{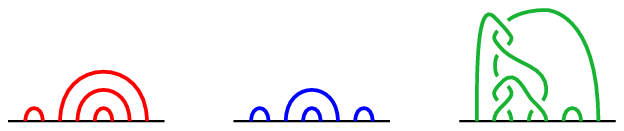}\vspace{-.4cm}
  \label{fig:b7}
  \caption{Result of braid and Reidemeister moves to concentrate diagram}
\end{subfigure}
	\caption{$7^{0,-2}_{1}$}
\label{fig:b}
\end{figure}

\begin{figure}[h!]
\begin{subfigure}{.3\textwidth}
  \centering
  \includegraphics[width=.7\linewidth]{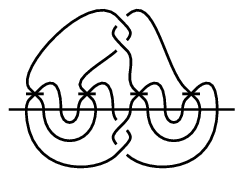}
  \label{fig:c1}
  \caption{ch-diagram\vspace{.4cm}}
\end{subfigure}
\begin{subfigure}{.7\textwidth}
  \centering
  \includegraphics[width=1\linewidth]{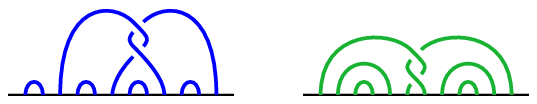}\vspace{-.4cm}
  \label{fig:c2}
  \caption{Result of converting ch-diagram to tri-plane diagram}
\end{subfigure}
	\caption{$8^{1,1}_1$}
\label{fig:c}
\end{figure}

\begin{figure}[h!]
\begin{subfigure}{.3\textwidth}
  \centering
  \includegraphics[width=.7\linewidth]{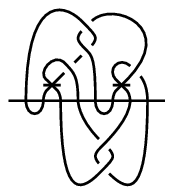}
  \label{fig:d1}
  \caption{ch-diagram\vspace{.4cm}}
\end{subfigure}
\begin{subfigure}{.7\textwidth}
  \centering
  \includegraphics[width=1\linewidth]{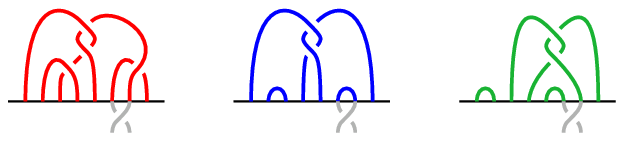}\vspace{-.4cm}
  \label{fig:d2}
  \caption{Result of converting ch-diagram to tri-plane diagram\vspace{.4cm}}
\end{subfigure}
\begin{subfigure}{.7\textwidth}
  \centering
  \includegraphics[width=1\linewidth]{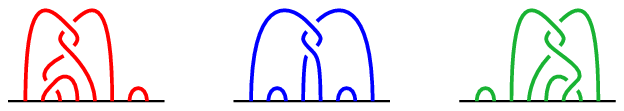}\vspace{-.4cm}
  \label{fig:d3}
  \caption{Crossing number lowered by braid and Reidemeister moves}
\end{subfigure}
	\caption{$8^{-1,-1}_1$}
\label{fig:d}
\end{figure}

\begin{figure}[h!]
\begin{subfigure}{.3\textwidth}
  \centering
  \includegraphics[width=.7\linewidth]{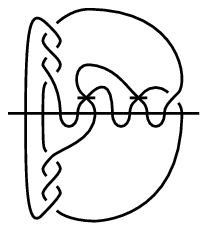}
  \label{fig:e1}
  \caption{ch-diagram\vspace{.4cm}}
\end{subfigure}
\begin{subfigure}{.7\textwidth}
  \centering
  \includegraphics[width=1\linewidth]{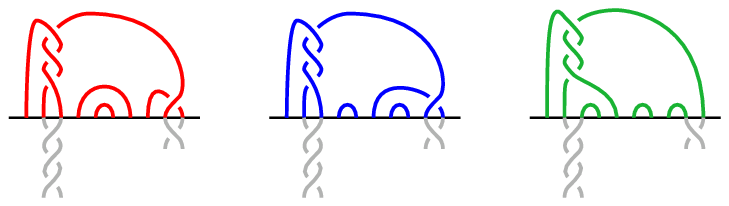}\vspace{-.4cm}
  \label{fig:e2}
  \caption{Result of converting ch-diagram to tri-plane diagram\vspace{.4cm}}
\end{subfigure}
\begin{subfigure}{.7\textwidth}
  \centering
  \includegraphics[width=1\linewidth]{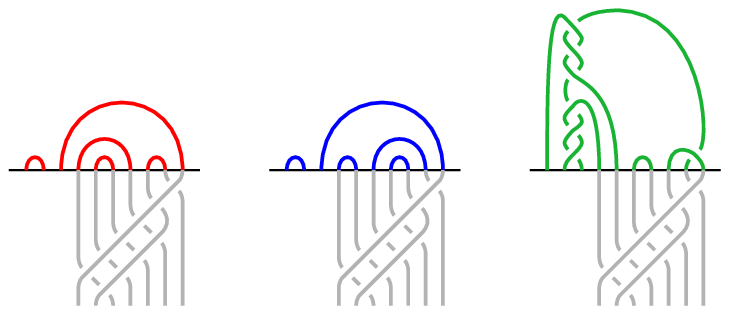}\vspace{-.4cm}
  \label{fig:e4}
  \caption{Crossing number lowered by braid and Reidemeister moves\vspace{.4cm}}
\end{subfigure}
\begin{subfigure}{.7\textwidth}
  \centering
  \includegraphics[width=1\linewidth]{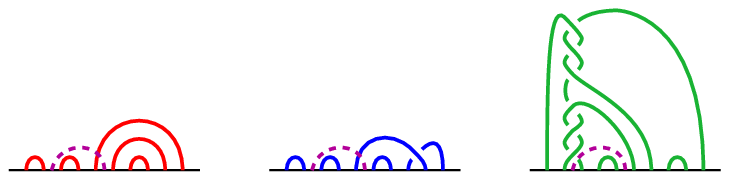}\vspace{-.4cm}
  \label{fig:e5}
  \caption{Braid and Reidemeister moves set up a destabilization\vspace{.4cm}}
\end{subfigure}
\begin{subfigure}{.7\textwidth}
  \centering
  \includegraphics[width=1\linewidth]{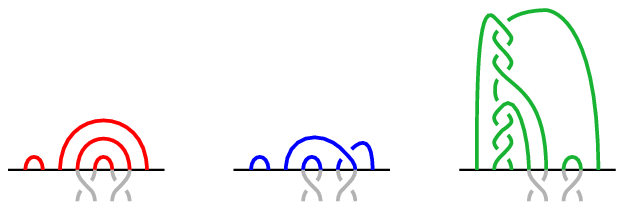}\vspace{-.4cm}
  \label{fig:e6}
  \caption{Result of destabilization\vspace{.4cm}}
\end{subfigure}
\begin{subfigure}{.7\textwidth}
  \centering
  \includegraphics[width=1\linewidth]{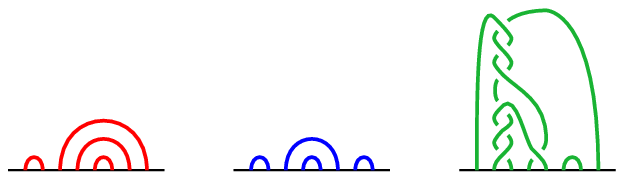}\vspace{-.4cm}
  \label{fig:e7}
  \caption{Result of braid and Reidemeister moves to concentrate diagram}
\end{subfigure}
	\caption{$9_1$}
\label{fig:e}
\end{figure}

\begin{figure}[h!]
\begin{subfigure}{.3\textwidth}
  \centering
  \includegraphics[width=.7\linewidth]{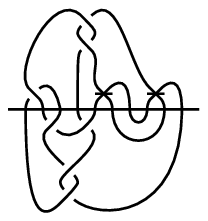}
  \label{fig:f1}
  \caption{ch-diagram\vspace{.4cm}}
\end{subfigure}
\begin{subfigure}{.7\textwidth}
  \centering
  \includegraphics[width=1\linewidth]{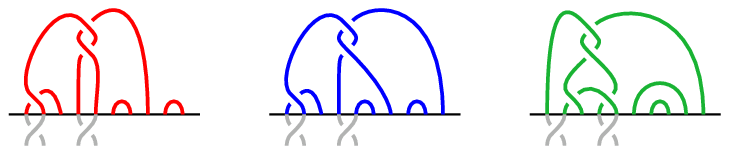}\vspace{-.4cm}
  \label{fig:f2}
  \caption{Result of converting ch-diagram to tri-plane diagram\vspace{.4cm}}
\end{subfigure}
\begin{subfigure}{.7\textwidth}
  \centering
  \includegraphics[width=1\linewidth]{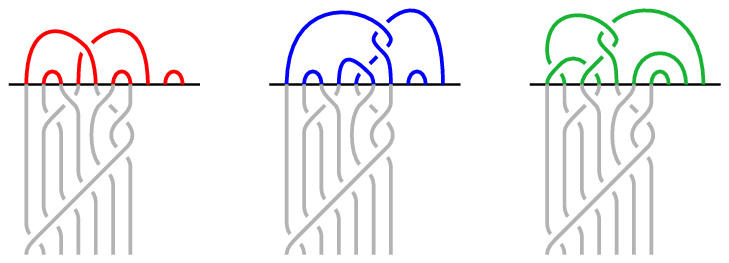}\vspace{-.4cm}
  \label{fig:f3}
  \caption{Crossing number lowered by braid and Reidemeister moves\vspace{.4cm}}
\end{subfigure}
\begin{subfigure}{.7\textwidth}
  \centering
  \includegraphics[width=1\linewidth]{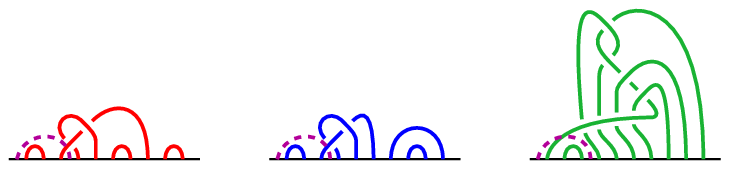}\vspace{-.4cm}
  \label{fig:f4}
  \caption{Braid and Reidemeister moves set up a destabilization\vspace{.4cm}}
\end{subfigure}
\begin{subfigure}{.7\textwidth}
  \centering
  \includegraphics[width=1\linewidth]{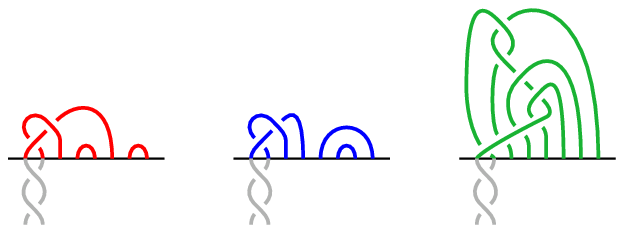}\vspace{-.4cm}
  \label{fig:f5}
  \caption{Result of destabilization\vspace{.4cm}}
\end{subfigure}
\begin{subfigure}{.7\textwidth}
  \centering
  \includegraphics[width=1\linewidth]{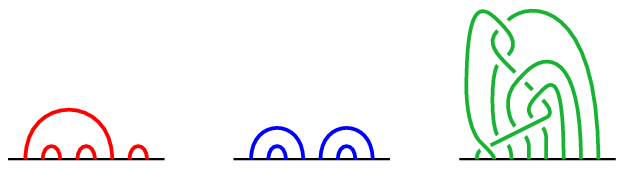}\vspace{-.4cm}
  \label{fig:f6}
  \caption{Result of braid and Reidemeister moves to concentrate diagram}
\end{subfigure}
	\caption{$9^{0,1}_1$}
\label{fig:f}
\end{figure}

\begin{figure}[h!]
\begin{subfigure}{.3\textwidth}
  \centering
  \includegraphics[width=.7\linewidth]{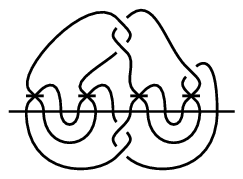}
  \label{fig:g1}
  \caption{ch-diagram\vspace{.4cm}}
\end{subfigure}
\begin{subfigure}{.7\textwidth}
  \centering
  \includegraphics[width=1\linewidth]{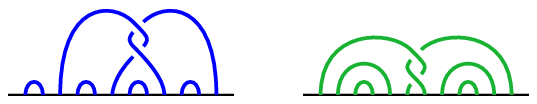}\vspace{-.4cm}
  \label{fig:g2}
  \caption{Result of converting ch-diagram to tri-plane diagram}
\end{subfigure}
	\caption{$9^{1,-2}_1$}
\label{fig:g}
\end{figure}

\begin{figure}[h!]
\begin{subfigure}{.3\textwidth}
  \centering
  \includegraphics[width=.7\linewidth]{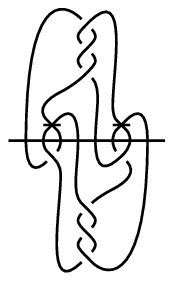}
  \label{fig:h1}
  \caption{ch-diagram\vspace{.4cm}}
\end{subfigure}
\begin{subfigure}{.7\textwidth}
  \centering
  \includegraphics[width=1\linewidth]{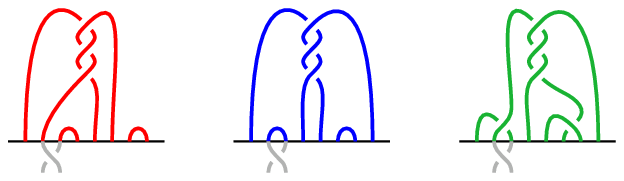}\vspace{-0.4cm}
  \label{fig:h2}
  \caption{Result of converting ch-diagram to tri-plane diagram\vspace{0.4cm}}
\end{subfigure}
\begin{subfigure}{.7\textwidth}
  \centering
  \includegraphics[width=1\linewidth]{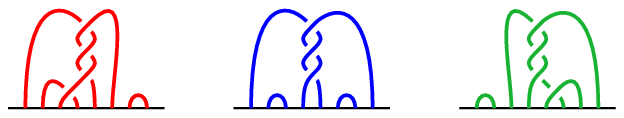}\vspace{-.4cm}
  \label{fig:h3}
  \caption{Crossing number lowered by braid and Reidemeister moves}
\end{subfigure}
	\caption{$10_2$}
\label{fig:h}
\end{figure}

\begin{figure}[h!]
\begin{subfigure}{.3\textwidth}
  \centering
  \includegraphics[width=.7\linewidth]{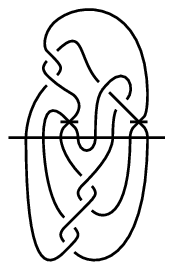}
  \label{fig:i1}
  \caption{ch-diagram\vspace{.4cm}}
\end{subfigure}
\begin{subfigure}{.7\textwidth}
  \centering
  \includegraphics[width=1\linewidth]{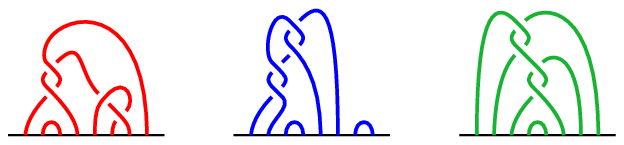}\vspace{-0.4cm}
  \label{fig:i2}
  \caption{Result of converting ch-diagram to tri-plane diagram\vspace{0.4cm}}
\end{subfigure}
\begin{subfigure}{.7\textwidth}
  \centering
  \includegraphics[width=1\linewidth]{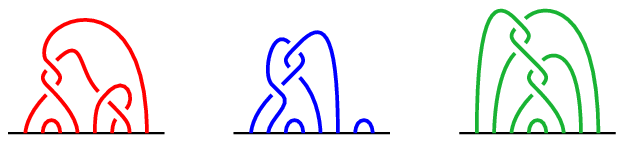}\vspace{-.4cm}
  \label{fig:h3}
  \caption{Crossing number lowered by Reidemeister moves}
\end{subfigure}
	\caption{$10_3$}
\label{fig:i}
\end{figure}

\begin{figure}[h!]
\begin{subfigure}{.3\textwidth}
  \centering
  \includegraphics[width=.7\linewidth]{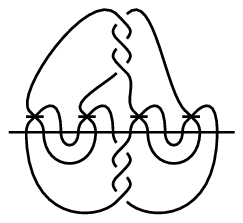}
  \label{fig:j1}
  \caption{ch-diagram\vspace{.4cm}}
\end{subfigure}
\begin{subfigure}{.7\textwidth}
  \centering
  \includegraphics[width=1\linewidth]{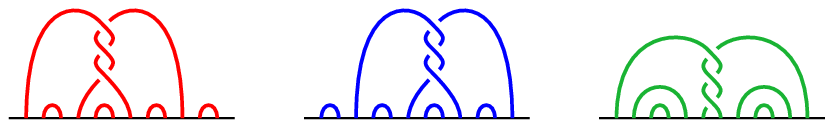}\vspace{-.4cm}
  \label{fig:j2}
  \caption{Result of converting ch-diagram to tri-plane diagram}
\end{subfigure}
	\caption{$10^1_1$}
\label{fig:j}
\end{figure}

\begin{figure}[h!]
\begin{subfigure}{.3\textwidth}
  \centering
  \includegraphics[width=.7\linewidth]{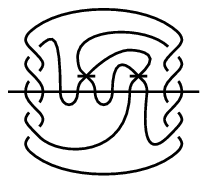}
  \label{fig:k1}
  \caption{ch-diagram\vspace{.4cm}}
\end{subfigure}
\begin{subfigure}{.7\textwidth}
  \centering
  \includegraphics[width=1\linewidth]{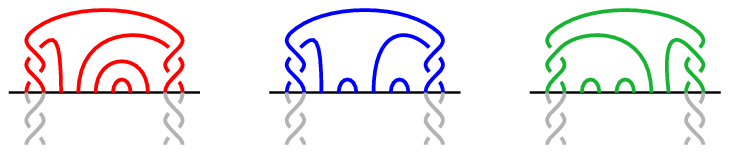}\vspace{-0.4cm}
  \label{fig:k2}
  \caption{Result of converting ch-diagram to tri-plane diagram\vspace{0.4cm}}
\end{subfigure}
\begin{subfigure}{.7\textwidth}
  \centering
  \includegraphics[width=1\linewidth]{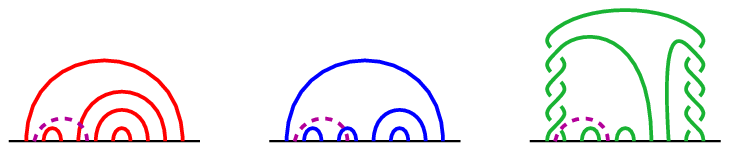}\vspace{-.4cm}
  \label{fig:k3}
  \caption{Crossing number lowered by braid and Reidemeister moves}
\end{subfigure}
\begin{subfigure}{.7\textwidth}
  \centering
  \includegraphics[width=1\linewidth]{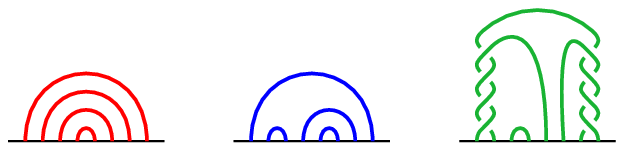}\vspace{-.4cm}
  \label{fig:k4}
  \caption{Result of destabilization\vspace{.4cm}}
\end{subfigure}
	\caption{$10^{0,1}_1$}
\label{fig:k}
\end{figure}

\begin{figure}[h!]
\begin{subfigure}{.3\textwidth}
  \centering
  \includegraphics[width=.7\linewidth]{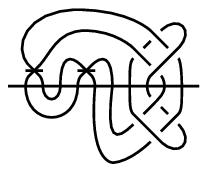}
  \label{fig:l1}
  \caption{ch-diagram\vspace{.4cm}}
\end{subfigure}
\begin{subfigure}{.7\textwidth}
  \centering
  \includegraphics[width=1\linewidth]{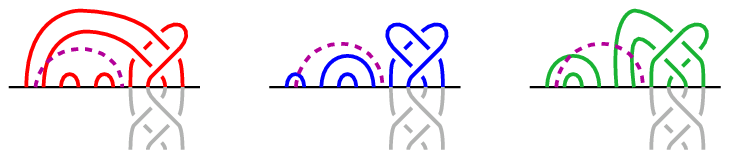}\vspace{-0.4cm}
  \label{fig:l2}
  \caption{Result of converting ch-diagram to tri-plane diagram\vspace{0.4cm}}
\end{subfigure}
\begin{subfigure}{.7\textwidth}
  \centering
  \includegraphics[width=1\linewidth]{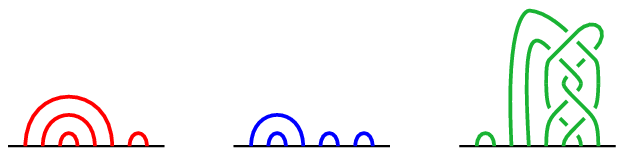}\vspace{-.4cm}
  \label{fig:l3}
  \caption{Destabilization and crossing number lowered by braid and Reidemeister moves}
\end{subfigure}
	\caption{$10^{0,1}_2$}
\label{fig:l}
\end{figure}

\begin{figure}[h!]
\begin{subfigure}{.3\textwidth}
  \centering
  \includegraphics[width=.7\linewidth]{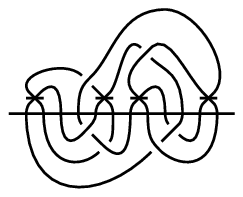}
  \label{fig:m1}
  \caption{ch-diagram\vspace{.4cm}}
\end{subfigure}
\begin{subfigure}{.7\textwidth}
  \centering
  \includegraphics[width=1\linewidth]{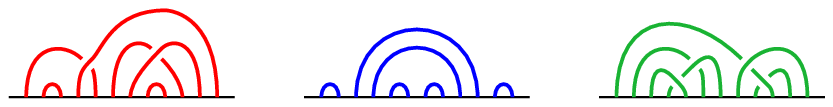}\vspace{-.4cm}
  \label{fig:m2}
  \caption{Result of converting ch-diagram to tri-plane diagram}
\end{subfigure}
	\caption{$10^{1,1}_1$}
\label{fig:m}
\end{figure}

\begin{figure}[h!]
\begin{subfigure}{.3\textwidth}
  \centering
  \includegraphics[width=.7\linewidth]{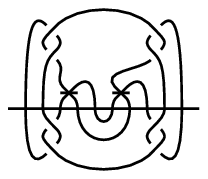}
  \label{fig:n1}
  \caption{ch-diagram\vspace{.4cm}}
\end{subfigure}
\begin{subfigure}{.7\textwidth}
  \centering
  \includegraphics[width=1\linewidth]{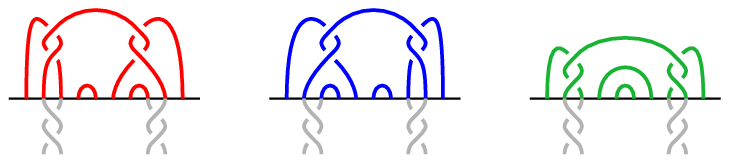}\vspace{-0.4cm}
  \label{fig:n2}
  \caption{Result of converting ch-diagram to tri-plane diagram\vspace{0.4cm}}
\end{subfigure}
\begin{subfigure}{.7\textwidth}
  \centering
  \includegraphics[width=1\linewidth]{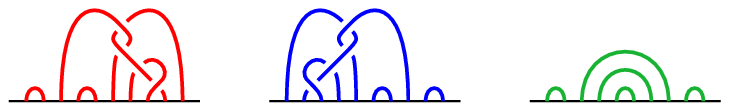}\vspace{-.4cm}
  \label{fig:n3}
  \caption{Crossing number lowered by braid and Reidemeister moves}
\end{subfigure}
	\caption{$10^{0,0,1}_1$}
\label{fig:n}
\end{figure}

\begin{figure}[h!]
\begin{subfigure}{.6\textwidth}
  \centering
  \includegraphics[width=.325\linewidth]{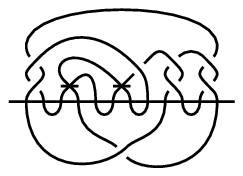} \\
  \label{fig:o1}
  \caption{ch-diagram\vspace{.4cm}}
\end{subfigure}
\begin{subfigure}{.65\textwidth}
  \centering
  \includegraphics[width=1\linewidth]{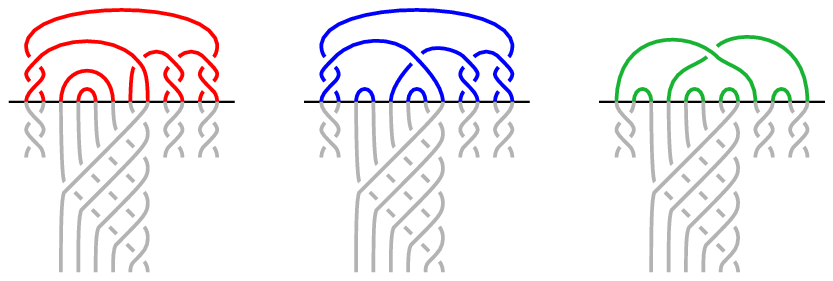}\vspace{-.4cm}
  \label{fig:o2}
  \caption{Result of converting ch-diagram to tri-plane diagram\vspace{.4cm}}
\end{subfigure}
\begin{subfigure}{.65\textwidth}
  \centering
  \includegraphics[width=1\linewidth]{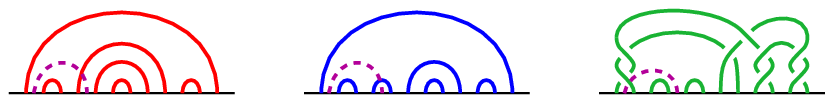}\vspace{-.4cm}
  \label{fig:o3}
  \caption{Crossing number lowered by braid and Reidemeister moves\vspace{.4cm}}
\end{subfigure}
\begin{subfigure}{.65\textwidth}
  \centering
  \includegraphics[width=1\linewidth]{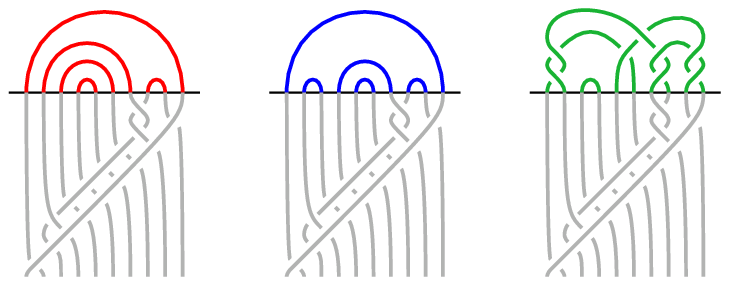}\vspace{-.4cm}
  \label{fig:o4}
  \caption{Result of destabilization\vspace{.4cm}}
\end{subfigure}
\begin{subfigure}{.65\textwidth}
  \centering
  \includegraphics[width=1\linewidth]{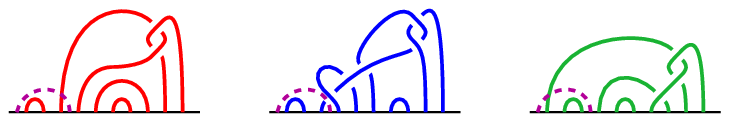}\vspace{-.4cm}
  \label{fig:o5}
  \caption{Braid and Reidemeister moves set up another destabilization\vspace{.4cm}}
\end{subfigure}
\begin{subfigure}{.65\textwidth}
  \centering
  \includegraphics[width=1\linewidth]{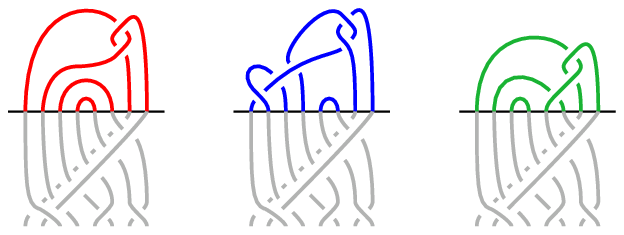}\vspace{-.4cm}
  \label{fig:o6}
  \caption{Result of destabilization}
\end{subfigure}
\begin{subfigure}{1\textwidth}
  \centering
  \includegraphics[width=.65\linewidth]{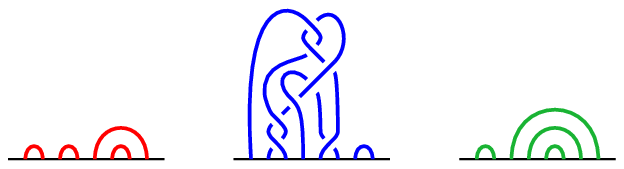}\vspace{-.2cm}
  \label{fig:o7}
  \caption{Result of braid and Reidemeister moves to decrease crossing number}
\end{subfigure}
	\caption{$10^{0,-2}_1$}
\label{fig:o}
\end{figure}

\begin{figure}[h!]
\begin{subfigure}{.3\textwidth}
  \centering
  \includegraphics[width=.7\linewidth]{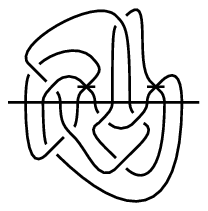}
  \label{fig:p1}
  \caption{ch-diagram\vspace{.4cm}}
\end{subfigure}
\begin{subfigure}{.7\textwidth}
  \centering
  \includegraphics[width=1\linewidth]{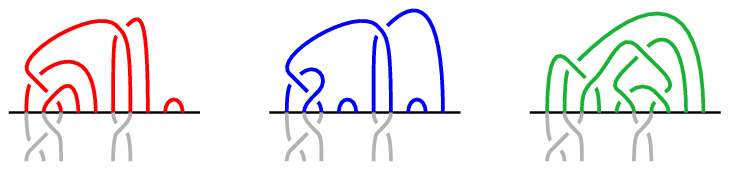}\vspace{-.4cm}
  \label{fig:p2}
  \caption{Result of converting ch-diagram to tri-plane diagram\vspace{.4cm}}
\end{subfigure}
\begin{subfigure}{.7\textwidth}
  \centering
  \includegraphics[width=1\linewidth]{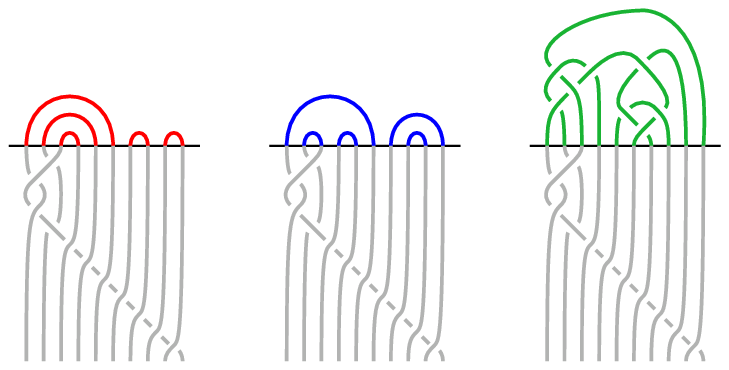}\vspace{-.4cm}
  \label{fig:p3}
  \caption{Minimum realized crossing number\vspace{.4cm}}
\end{subfigure}
\begin{subfigure}{.7\textwidth}
  \centering
  \includegraphics[width=1\linewidth]{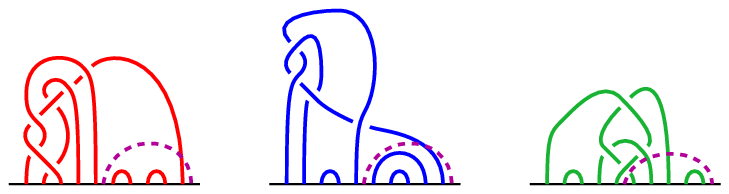}\vspace{-.4cm}
  \label{fig:p4}
  \caption{Braid and Reidemeister moves set up a destabilization\vspace{.4cm}}
\end{subfigure}
\begin{subfigure}{.7\textwidth}
  \centering
  \includegraphics[width=1\linewidth]{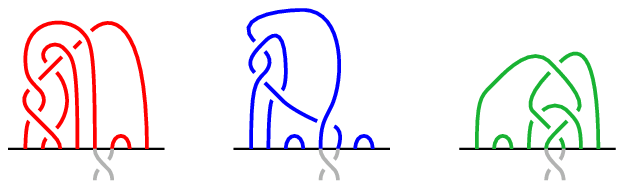}\vspace{-.4cm}
  \label{fig:p5}
  \caption{Result of destabilization\vspace{.4cm}}
\end{subfigure}
\begin{subfigure}{1\textwidth}
  \centering
  \includegraphics[width=.7\linewidth]{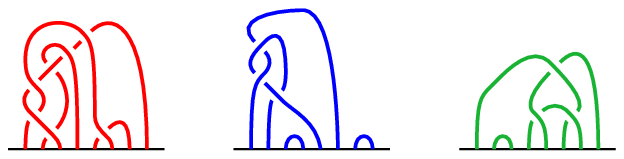}\vspace{-.1cm}
  \label{fig:p6}
  \caption{Result of braid and Reidemeister moves to decrease crossing number}
\end{subfigure}
	\caption{$10^{0,-2}_2$}
\label{fig:p}
\end{figure}

\begin{figure}[h!]
\begin{subfigure}{.3\textwidth}
  \centering
  \includegraphics[width=.7\linewidth]{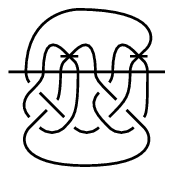}
  \label{fig:q1}
  \caption{ch-diagram\vspace{.4cm}}
\end{subfigure}
\begin{subfigure}{.7\textwidth}
  \centering
  \includegraphics[width=1\linewidth]{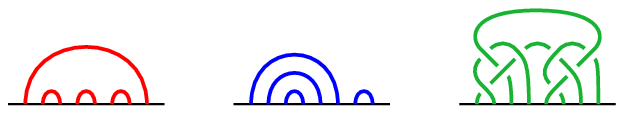}\vspace{-.4cm}
  \label{fig:q2}
  \caption{Result of converting ch-diagram to tri-plane diagram}
\end{subfigure}
	\caption{$10^{-1,-1}_1$}
\label{fig:q}
\end{figure}

\begin{figure}[h!]
\begin{subfigure}{.3\textwidth}
  \centering
  \includegraphics[width=.7\linewidth]{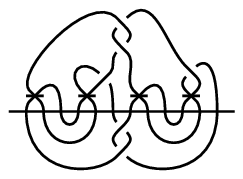}
  \label{fig:r1}
  \caption{ch-diagram\vspace{.4cm}}
\end{subfigure}
\begin{subfigure}{.7\textwidth}
  \centering
  \includegraphics[width=1\linewidth]{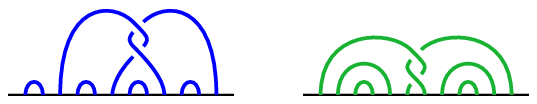}\vspace{-.4cm}
  \label{fig:r2}
  \caption{Result of converting ch-diagram to tri-plane diagram}
\end{subfigure}
	\caption{$10^{-2,-2}_1$}
\label{fig:r}
\end{figure}

\section{Questions}\label{sec:questions}

Following Remark~\ref{degen}, there are examples of surfaces for which both crossing number and bridge number degenerate under the connected sum operation, but we know no such examples among \emph{orientable} surfaces.

\begin{question}
For orientable surfaces $\K$ and $\K'$, does it hold that
\[ c(\K \# \K') = c(\K) + c(\K') \quad \text{ and } \quad b(\K \# \K') = b(\K) + b(\K') - 1?\]
\end{question}

For classical knots, the analogous equality for bridge number is known to be true~\cite{schubert,schultens}, while the question for crossing number is a notoriously difficult open problem in general.

In Table~\ref{table1}, we observe that we have minimized crossing number via concentrated diagrams for only about half of the surfaces $\K$.  We can define the \emph{concentrated crossing number} of $\K$, denoted $c^c(\K)$, to be the minimum crossing number among concentrated diagrams, leading naturally to the next question.

\begin{question}
Does there exist a surface $\K$ such that $c(\K) < c^c(\K)$?  If so, what is the largest possible difference $c^c(\K) - c(\K)$?
\end{question}

As noted above, the surface $10^{0,-2}_1$ provides an interesting example in that it has both an 8-crossing 5-bridge diagram and a 13-crossing 4-bridge diagram, where we were not able to reduce the crossing number below 13 among the family of 4-bridge diagrams for $10^{0,-2}_1$.

\begin{question}
Does there exist a surface $\K \subset S^4$ such that $b(\K)$ and $c(\K)$ cannot be realized by a single diagram?
\end{question}

Using the data in Table~\ref{table1}, it appears that crossing number is bounded below by bridge number, although this is not the case for unknotted surfaces, since $b(\Pp^{n,m}) = n+m+1$.

\begin{question}
Aside from constructions obtained by taking the connected sum with unknotted surfaces, does there exist a knotted surface $\K \subset S^4$ such that $b(\K) > c(\K)$?
\end{question}

Finally, we note that every surface $\K$ in our table satisfies $e(\K) = 0$.  By taking a connected sum of $\Pp^{\pm}$ with any $\K$ from the table, we can obtain a surface with $e(\Pp^{\pm} \# \K) = \pm 2$ and $c(P^{\pm} \# \K) \leq c(\K) + 1$.  But the following remains unknown for surfaces that are not connected sums.

\begin{question}
What is the smallest crossing number of a surface $\K$ such that $\K \neq \K' \# \Pp^{\pm}$ for some other surface $\K'$ and such that $e(\K) \neq 0$?
\end{question}

\bibliographystyle{amsalpha}
\bibliography{triplanebib}

\end{document}